\documentclass[12pt]{amsart}

\usepackage[top=100pt,bottom=60pt,left=96pt,right=92pt]{geometry}

\usepackage{hyperref}
\usepackage{amsmath,amsthm,amssymb,amsfonts,enumerate,xcolor,layout,tikz,subcaption,enumitem,fullpage, verbatim, colordvi}
\usetikzlibrary{calc}   % gives us abs(), arithmetic in coordinates
\usepackage{stmaryrd}
\usepackage{graphicx}
\newtheorem{theorem}{Theorem}[section]
\newtheorem*{theorem*}{Theorem}
\newtheorem{proposition}[theorem]{Proposition}

\newtheorem{lemma}[theorem]{Lemma}

\theoremstyle{definition}
\newtheorem{definition}[theorem]{Definition}
\newtheorem{remark}[theorem]{Remark}

\theoremstyle{plain}

\definecolor{linkblue}{rgb}{0,0,.6}
\definecolor{citered}{rgb}{.7,0,0}
\hypersetup{colorlinks =true, linkcolor=linkblue, citecolor = citered, urlcolor=linkblue}

\def\C{{\mathbb C}}

\def\P{{\mathbb P}}
\def\R{{\mathbb R}}

\newcommand{\De}{\Delta}

\def\Tilde{\widetilde}

\newcommand{\Hess}{\operatorname{Hess}}

\newcommand{\purple}[1]{\textbf{\purple{#1}}}

\begin{document}

\title{Semitoric Families on Pentagon Spaces}
\author{Yichen Liu}
\address{Department of Mathematics \\
	University of Illinois at Urbana-Champaign\\Champaign, Illinois 61801}
\curraddr{}
\email{yichen23@illinois.edu}
\thanks{}

\author{Aerim Si}
\address{Department of Mathematics\\
	University of Illinois at Urbana-Champaign\\Champaign, Illinois 61801}
\curraddr{}
\email{aerimsi2@illinois.edu}
\thanks{}

\date{\today}

\begin{abstract}
Semitoric systems are a special type of 4-dimensional integrable system where one of the functions is the moment map of a Hamiltonian $S^1$-action. While their classification is well understood thanks to the work of Pelayo and V{\~u} Ng\d{o}c, relatively few explicit examples are known. Recently, Le Floch and Palmer introduced semitoric transition families in which a singular point transitions between elliptic-elliptic type and focus-focus type as the parameter varies. In this paper, we construct new semitoric transition families on pentagon spaces by interpolating two semitoric systems of toric type. More precisely, we exhibit semitoric transition families, each with at least one transition point, on pentagon spaces of different diffeotypes, all defined by the same explicit interpolation formula.
\end{abstract}

\maketitle

%\tableofcontents

\section{Introduction}\label{sec:intro}
Let $(M,\omega)$ be a symplectic manifold. Unless otherwise specified, we always assume that $M$ is connected. Recall that to each smooth function $f: M \to \R$, one can associate its {\bf Hamiltonian vector field} $X_f$ given by the equation $df = \omega(X_f, \cdot)$. Given two functions $f,g: M \to \R$, their Poisson bracket $\{f,g\}$ is another function defined by $\{f,g\} := \omega(X_f,X_g)$. We say two functions $f,g$ {\bf Poisson commute} if $\{f,g\}=0$. 

An {\bf integrable system} on a $2n$-dimensional symplectic manifold $(M,\omega)$ is a map $F=(f_1,\ldots,f_n): M \to \R^n$ such that $\{f_i,f_j\}=0$ for all $i,j$ and $df_1,\ldots,df_n$ are linearly independent almost everywhere. A {\bf semitoric system} is a special kind of integrable system when $\dim M= 4$. More specifically, an integrable system $(M,\omega, F=(J,H))$ is a semitoric system if $J$ is a proper map that generates an effective Hamiltonian $S^1$-action on $M$ and the momentum map $F$ only has non-degenerate singularities without hyperbolic type (see Definition~\ref{def:semitoric}). Semitoric systems were introduced by ~\cite{VN2007} following the work of~\cite{Sy2003}. Pelayo and V{\~u} Ng\d{o}c~\cite{PeVN2009,PeVN2011} first classified simple semitoric systems. Following their work, Pelayo-Palmer-Tang~\cite{PPT24} generalized the classification to any semitoric system. Semitoric systems have been intensively studied in recent years. See~\cite{ADH-momenta,ADH-spin-osc,AHP-twist,LFP} for studies of invariants of semitoric systems and \cite{AH-survey,SV17,D23,Palmer25} for recent survey articles and additional references.

The coupled angular momenta system was the first example of a family of semitoric systems known to us. It was first introduced in~\cite{SaZh1999} and studied in detail in~\cite{LFP}. More families were studied later by many people; see~\cite{HoPa2017,DP,HohMeu,LFP24,LFP,LFP-fam2}. In~\cite{LFP24}, the authors gave the definition of semitorc families with degenerate times and semitoric transition families. A \textbf{semitoric family} with \textit{degenerate times} $t_1, \dots, t_n \in [0,1]$, is a family of integrable systems $(M,\omega,F_t)$, $0\leq t \leq 1$ on a 4-dimensional symplectic manifold $(M,\omega)$ such that $(M,\omega,F_t:=(J,H_t))$, where $H_t$ is of the form $H_t=H(t,\cdot)$ for some smooth $H:[0,1] \times M \rightarrow \R$, is semitoric if and only if $t \in [0,1] \setminus \{t_1, \dots, t_n\}$. In this paper, we study semitoric transition families on pentagon spaces.

A \textbf{semitoric transition family} with \textit{transition point} $p \in M$ and \textit{transition times} $t^-, t^+ $, with $0<t^-<t^+<1$, is a family of integrable systems $(M,\omega,F_t=(J,H_t))$, $0\leq t \leq 1$ on a 4-dimensional symplectic manifold $(M, \omega)$  such that

\begin{itemize}
    \item $(M,\omega,F_t)$ is semitoric for $t \in [0,1] \setminus \{t^-,t^+\}$;
    
    \item $p$ is a singular point of elliptic-elliptic type of $(M,\omega,F_t)$ for $t\in[0,t^-)\cup(t^+,1]$;
 
    \item $p$ is a singular point of focus-focus type of $(M,\omega,F_t)$ for $t\in(t^-,t^+)$;

    \item there are no degenerate singular points in $M \setminus \{p\}$ for $t=t^\pm$; and

    \item if $p$ is a maximum (respectively, minimum) of $(H_0)|_{J^{-1}(J(p))}$ then $p$ is a minimum (respectively, maximum) of $(H_1)|_{J^{-1}(J(p))}$.
     
\end{itemize}

Let $r_1,\ldots,r_n \in \R_+$ be positive numbers. Define the space of $n$-gons with side lengths $r_1,\ldots,r_n$ to be the set 
\[\widetilde{\mathcal{P}}(r_1,\ldots,r_n):= \{\rho: \{1,\ldots,n\} \to \R^3 \mid \sum_{i=1}^n \rho_i = 0, |\rho_i| = r_i\}.\]
The {\bf $\mathbf{n}$-polygon space} $\mathcal{P}(r_1,\ldots,r_n)$ is the quotient space $\mathcal{P}(r_1,\ldots,r_n):= \widetilde{\mathcal{P}}(r_1,\ldots,r_n)/SO(3)$. In other words, the polygon space consists of the configurations in $\R^3$ of polygons with side lengths $r_1,\ldots,r_n$, up to rotation. We say $r_1,\ldots,r_n$ are \emph{generic} if there does not exist a map $\epsilon:\{1,\ldots,n\} \to \{\pm1\}$ such that $\sum_{i=1}^n \epsilon_i r_i=0$. Generic side lengths guarantee that the polygon space, if not empty, is a smooth manifold of dimension $2n-6$. These spaces come with symplectic structures. Consider $\Tilde M: = S^2_{r_1} \times \cdots \times S^2_{r_n}$, the product of spheres of radii $r_1,\ldots,r_n$, with symplectic form $\frac{1}{r_1^2}\omega \oplus \cdots \oplus \frac{1}{r_n^2} \omega$, where $\omega_u(v,w) = u \cdot (v \times w)$. Here, we think of $S^2_{r_i}$ as a submanifold of $\R^3$. The map $\mu: \Tilde M \to \R^3$ given by $\mu(\rho_1,\ldots,\rho_n) = \sum_{i=1}^n \rho_i$ is a moment map of the diagonal Hamiltonian $SO(3)$-action on $\Tilde M$, where we identify $\mathfrak{so}_3^*$ with $\R^3$ equipped with the cross product. Then the symplectic reduction at the level $0$ gives us $\mathcal{P}(r_1,\ldots,r_n) = \mu^{-1}(0)/SO(3)$, so it comes with a symplectic structure. See~\cite{HK97,Kl94,KM95,KM96}.
Let $I \subset \{1,\ldots,n\}$ be a nonempty proper subset. Define $\ell_I: \mathcal{P}(r_1,\ldots,r_n) \to \R$ by $\ell_I([\rho])=|\sum_{i \in I} \rho_i|$.

Before stating our result, we first define two important pentagons $P,Q$ that will be transition points of $F_t$ and numbers $t_P^\pm, t_Q^\pm$ that will be transition times. For positive numbers $r_1,\ldots,r_5$, we define $j_P = r_3+r_4-r_5$, $j_Q = -r_3+r_4+r_5$.
The pentagon $P \in \mathcal{P}(r_1,r_2,r_3,r_4,r_5)$ is represented by the planar pentagon whose five edges are given by 
    \begin{gather*}
        \rho_1 = \left(\frac{r_1^2-r^2_2 +j_P^2}{2j_P}, \frac{\sqrt{\big((r_1+j_P)^2-r_2^2\big)\big(r_2^2-(r_1-j_P)^2\big)}}{2j_P},0 \right), \\ \rho_2 = \left(\frac{r_2^2-r_1^2 +j_P^2}{2j_P}, -\frac{\sqrt{\big((r_1+j_P)^2-r_2^2\big)\big(r_2^2-(r_1-j_P)^2\big)}}{2j_P},0 \right),\\ \rho_3 = (-r_3,0,0), \rho_4 =(-r_4,0,0), \rho_5 = (r_5,0,0).
    \end{gather*}
      
The pentagon $Q \in \mathcal{P}(r_1,r_2,r_3,r_4,r_5)$ is represented by the planar pentagon whose five edges are given by 
    \begin{gather*}
        \rho_1 = \left(\frac{r_1^2-r^2_2 +j_Q^2}{2j_Q}, \frac{\sqrt{\big((r_1+j_Q)^2-r_2^2\big)\big(r_2^2-(r_1-j_Q)^2\big)}}{2j_Q},0 \right), \\ \rho_2 = \left(\frac{r_2^2-r_1^2 +j_Q^2}{2j_Q}, -\frac{\sqrt{\big((r_1+j_Q)^2-r_2^2\big)\big(r_2^2-(r_1-j_Q)^2\big)}}{2j_Q},0 \right),\\ \rho_3 = (r_3,0,0), \rho_4 =(-r_4,0,0), \rho_5 = (-r_5,0,0).
    \end{gather*}
The quantities $t_P^\pm, t_Q^\pm$ are defined as follows.
\[t_P^\pm = \frac{\,r_4j_P+r_3r_5+(r_4-r_5)^2\,\pm\,2\sqrt{r_{3}r_{4}r_{5}j_P}}
       {\, (r_{3}+r_{5})^{2}+4r_{4}j_P \,},\quad t_Q^\pm = \frac{\,r_4j_Q+r_{3}r_{5}+(r_4+r_5)^2\,\pm\,2\sqrt{r_{3}r_{4}r_{5}j_Q}}
       {\, (r_{3}+r_{5})^{2}+4r_{4}j_Q \,}.\] 
    Notice that $P,Q$ are not always well-defined. More specifically, $P$ is well-defined in case (P),(P+Q), but not in case (Q) in Theorem~\ref{thm:main} below and $Q$ is well-defined in case (Q),(P+Q), but not in case (P). Similarly, $t_P^{\pm}$ is a well-defined real number in case (P),(P+Q), but may not be a real number in case (Q) and vice versa.

\begin{theorem}\label{thm:main}
    Let $r_1,\ldots,r_5$ be positive numbers. Consider the family $F_t:=(\ell_{12},t\ell_{34}^2 +(1-t)\ell_{45}^2)$ on $\mathcal{P}(r_1,r_2,r_3,r_4,r_5)$.
    \begin{enumerate}
        \item[(P)] If $|r_3-r_4-r_5| < |r_1-r_2| < r_3+r_4-r_5$ and $r_3+r_4-r_5 < r_1+r_2 < r_3-r_4+r_5$, $F_t$ is a semitoric transition family with transition point $P$ and transition times $t_P^\pm$.
        \item[(Q)] If $|r_3+r_4-r_5| < |r_1-r_2| < -r_3+r_4+r_5$ and $-r_3+r_4+r_5 < r_1+r_2 < r_3-r_4+r_5$, $F_t$ is a semitoric transition family with transition point $Q$ and transition times $t_Q^\pm$.
        \item[(P+Q)] If $-r_4+|r_3-r_5| < |r_1-r_2| < r_4-|r_3-r_5|$, $r_4+|r_3-r_5| < r_1+r_2 < r_3-r_4+r_5$ and $r_3 \neq r_5$, $F_t$ is a semitoric transition family with two distinct transition points $P$  and $Q$ with transition times $t_P^\pm$ and $t_Q^\pm$ respectively.        
        \end{enumerate}    
\end{theorem}
\begin{remark}
    In case $(P+Q)$, we assume that $r_3 \neq r_5$, which implies that $0 < t_P^- < t_P^+ < \frac{1}{2} < t_Q^{-} < t_Q^{+} <1$. When $r_3=r_5$, both $P$ and $Q$ are degenerate at the time $t_P^+ = t_Q^- = \frac{1}{2}$, so the fourth bullet point in the definition of a semitoric transition family is not satisfied. In this case, $P$ and $Q$ are in the same fiber of $(J,H_{\frac{1}{2}})$ which is a Lagrangian two-sphere.
\end{remark}

\begin{proof}[Proof of Theorem~\ref{thm:main}]
The proof of all three cases are almost identical, so we will only prove case (P).
We write $F_t = (J,H_t)$ and $M = \mathcal{P}(r_1,r_2,r_3,r_4,r_5)$. By assumption, the moment images of $(J,\sqrt{H_0})$ and $(J,\sqrt{H_1})$ are both Delzant polygons (see Figure~\ref{fig:moment-image}). Hence, for $t=0,1$, the system $F_t$ is of toric type.
Moreover, $H_1(P) = (r_3+r_4)^2$ is the global maximum of $H_1$ on $M$ and $H_0(P) = (r_4-r_5)^2$ is the global minimum of $H_0$ on $M$. Hence, it follows from Proposition~\ref{prop:transition-point} that $P$ is a transition point with transition times $t_P^\pm$.

We now show that for all $t \in [0,1]$, all singular points of $F_t$ other than $P$ are non-degenerate of either elliptic-elliptic or elliptic-regular type. Fix a singular point $m \in M \setminus \{P\}$ and define $c:=J(m)$. If $c = |r_1-r_2|$ or $c= r_1+r_2$, then $J^{-1}(c)$ is a fixed surface by Lemma~\ref{lem:fixed-components}, so Proposition~\ref{prop:reduction} implies that $J^{-1}(c)/S^1  = J^{-1}(c)\cong \mathcal{P}(c,r_3,r_4,r_5)$.  Proposition~\ref{prop:singular-points-on-fixed-surface} and Proposition~\ref{prop:rank-1} together imply that if $m$ is of rank one, then it is of elliptic-regular type, and if $m$ is of rank zero, then it is of elliptic-elliptic type.

Now assume that $c \neq |r_1-r_2|, r_1+r_2$. Since $m \neq P$, Lemma~\ref{lem:fixed-components} implies that the $S^1$-action generated by $J$ acts freely on $m$ and thus $m$ is  of rank one. Moreover, if we further assume $c \neq j_P$, then $c$ is a regular value of $J$ and $S^1$ acts freely on $J^{-1}(c)$. Hence, Proposition~\ref{prop:reduction} implies that $M^{\textrm{red},c}:= J^{-1}(c)/S^1 \cong \mathcal{P}(c,r_3,r_4,r_5)$. Let  $H_t^{\textrm{red},c}$ be the function on $M^{\textrm{red},c}$ to which $H_t$ descends, and $[m] \in M^{\textrm{red},c}$ be the point to which $m$ descends. The first part of Lemma~\ref{lem:LFP-rk1} implies that $[m]$ is a singular point of $H_t^{\textrm{red},c}$, which is an elliptic (in the Morse sense) singular point by Proposition~\ref{prop:rank-1}. Hence, by the second part of Lemma~\ref{lem:LFP-rk1}, $m$ is an elliptic-regular singularity of $F_t$.

Finally if $c=j_P$, then $M^{\textrm{red},c}$ is not a manifold. However, since $S^1$ acts freely on $m$, there is an $S^1$-invariant neighborhood around the orbit of $m$ on which $S^1$ acts freely. Hence, one can do symplectic reduction locally, and verify that $[m]$ is an elliptic singularity of $H_t^{\textrm{red},c}$ by computing its Hessian in the same way as in the proof of Proposition~\ref{prop:rank-1}. Therefore, any rank one points that are not on the fixed surfaces are of elliptic-regular type.

Therefore, we conclude that for $t \in [0,1] \setminus \{t_P^-,t_P^+ \}$, $(M,\omega,F_t)$ is a semitoric system, and when $t = t_P^\pm$, there are no degenerate singular points in $M \setminus \{P\}$.

\end{proof}

\begin{figure}
\centering
\begin{subfigure}{0.45\textwidth}
\centering
    \begin{tikzpicture}[scale=0.3]
    
    % Polygon
    \draw (0,1) -- (8,1) -- (8,9) -- (4,9) -- (0,5) -- (0,1);
    \draw[dashed] (4,4) -- (4,9);
    \node at (4,4) {$\times$};
\end{tikzpicture}
\caption{Case (P), $t \in (t_P^-,t_P^+)$}
\end{subfigure}
\begin{subfigure}{0.45\textwidth}
\centering
    \begin{tikzpicture}[scale=0.3]

    \draw (0,9) -- (8,9) -- (8,1) -- (4,1) -- (0,5) -- (0,9);
    \draw[dashed] (4,1) -- (4,4);
    \node at (4,4) {$\times$};
\end{tikzpicture}
\caption{Case (Q), $t \in (t_Q^-,t_Q^+)$}
\end{subfigure}
% \begin{subfigure}{0.35\textwidth}
% \centering
%     \begin{tikzpicture}[scale=0.3]
%     \draw (4,1) -- (8,1) -- (8,9) -- (4,9) -- (0,5) -- (0,3) -- (4,1);
%     \draw[dashed] (4,4) -- (4,9);
%     \node at (4,4) {$\times$};
% \end{tikzpicture}
% \caption{Case (P+Q), where $r_3=r_5$}
% \end{subfigure}\\
\begin{subfigure}{0.45\textwidth}
\centering
    \begin{tikzpicture}[scale=0.3]
    \draw (2,1) -- (8,1) -- (8,9) -- (4,9) -- (0,5) -- (0,3) -- (2,1);
    \draw[dashed] (4,5) -- (4,9);
    \node at (4,5) {$\times$};
    \node at (4,-1) {$t\in (t_P^-,t_P^+)$};
    \draw (12,1) -- (18,1) -- (18,9) -- (14,9) -- (10,5) -- (10,3) -- (12,1);
    \draw[dashed] (12,1) -- (12,3);
    \node at (12,3) {$\times$};
    \node at (14,-1) {$t\in (t_Q^-,t_Q^+)$};
\end{tikzpicture}
\caption{Case (P+Q), where $r_3>r_5$}
\end{subfigure}
\begin{subfigure}{0.45\textwidth}
\centering
    \begin{tikzpicture}[scale=0.3]
    \draw (2,9) -- (8,9) -- (8,1) -- (4,1) -- (0,5) -- (0,7) -- (2,9);
    \draw[dashed] (4,3) -- (4,1);
    \node at (4,3) {$\times$};
    \node at (4,-1) {$t\in (t_P^-,t_P^+)$};
    \draw (12,9) -- (18,9) -- (18,1) -- (14,1) -- (10,5) -- (10,7) -- (12,9);
    \draw[dashed] (12,9) -- (12,5);
    \node at (12,5) {$\times$};
    \node at (14,-1) {$t\in (t_Q^-,t_Q^+)$};
\end{tikzpicture}
\caption{Case (P+Q), where $r_3<r_5$}
\end{subfigure}
\caption{Each figure is a presentative of the semitoric polygon of the corresponding case when the focus-focus point is present.}
\label{fig:semitoric-polygons}
\end{figure}

\begin{figure}
\centering
\begin{subfigure}{0.23\textwidth}
\centering
    \begin{tikzpicture}[scale=0.3]
    \draw (0,1) -- (6,1) -- (8,3) -- (8,9) -- (4,9) -- (0,5) -- (0,1);
    \draw[dashed] (4,4) -- (4,9);
    \node at (4,4) {$\times$};
\end{tikzpicture}
\end{subfigure}
\begin{subfigure}{0.23\textwidth}
\centering
    \begin{tikzpicture}[scale=0.3]

    \draw (0,9) -- (8,9) -- (8,3) -- (6,1) -- (4,1) -- (0,5) -- (0,9);
    \draw[dashed] (4,1) -- (4,4);
    \node at (4,4) {$\times$};
\end{tikzpicture}
\end{subfigure}
\begin{subfigure}{0.23\textwidth}
\centering
    \begin{tikzpicture}[scale=0.3]
    \draw (2,1) -- (6,1) -- (8,3) -- (8,9) -- (4,9) -- (0,5) -- (0,3) -- (2,1);
    \draw[dashed] (4,4) -- (4,9);
    \node at (4,4) {$\times$};
\end{tikzpicture}
\end{subfigure}
\begin{subfigure}{0.23\textwidth}
\centering
    \begin{tikzpicture}[scale=0.3]
    \draw (2,9) -- (8,9) -- (8,3) -- (6,1) -- (4,1) -- (0,5) -- (0,7) -- (2,9);
    \draw[dashed] (4,4) -- (4,1);
    \node at (4,4) {$\times$};
\end{tikzpicture}
\end{subfigure}
\caption{Other semitoric polygons that can be obtained from the same family $F_t$ by changing the inequalities between the parameters $r_1,\dots,r_5$.}
\label{fig:remark}
\end{figure}

\begin{remark} 
Theorem~\ref{thm:main} gives new explicit examples of semitoric systems. In this remark, we explain possible generalizations of our constructions and how it is related to other works.
\begin{enumerate}
    \item The index sets in the statement are chosen for convenience. In fact, one may take any three distinct index sets $I,J,K$ of size $2$ such that $I$ is disjoint from both $J,K$ and consider the families $(\ell_I,t\ell_J^2+(1-t)\ell_K^2)$ on pentagon spaces whose side lengths satisfy constraints similar to those in Theorem~\ref{thm:main}. 
    \item Different choices of side lengths $r_1,\ldots,r_5$ that satisfy the same set of inequalities give different symplectic forms on pentagon spaces of the same diffeotype, so these families will be useful in understanding how the change of symplectic form affects the transition time and other invariants of these semitoric systems.
    \item Our construction can be generalized to $n$-polygon spaces $\mathcal{P}(r_1,\ldots,r_n)$ with $n >5$ to study higher-dimensional integrable systems where some components of the functions give a moment map of a Hamiltonian torus action. 
    \item The semitoric polygon invariant is one of the invariants of a semitoric system, analogous to the Delzant polytopes in the classification of symplectic toric manifolds. For the definition, we refer the reader to~\cite{VN2007} and~\cite[Section 4]{PeVN2009}. By \cite[Lemma 3.14]{LFP24}, a representative of the semitoric polygon (when the focus-focus point is present) of these families can be obtained by adding a marked point and a cut line to the semitoric polygon of the system $(J,H_1)$. We display these images in Figure~\ref{fig:semitoric-polygons}. 
    \item 
    We can obtain additional semitoric polygons shown in Figure~\ref{fig:remark} from the same family $F_t$ by changing the inequalities between the parameters $r_1,\dots,r_5$. But since the proof for the additional cases is essentially the same, we only list  the three cases (P),(Q), and (P+Q) for simplicity.    
     In fact, the other cases shown in Figure~\ref{fig:remark} also follow from \cite[Theorem 1.7]{LFP24} as these polygons are obtained by performing a toric type blowup. 
    \item By \cite{HK97}, pentagon spaces are generically symplectic toric four-manifolds, and there are in total five diffeotypes $\C\P^2,(S^2 \times S^2) \# k\overline{\C\P^2}$, $k=0,1,2,3$. Le Floch and Palmer proved that $\C\P^2$ does not admit a semitoric transition family (see~\cite[Proposition 5.2]{LFP-fam2} and the discussion there). Le Floch and Pelayo studied the semitoric transition family given by the coupled angular momenta system on $S^2 \times S^2$ in detail (see~\cite{LFP}). Our results provide new explicit families for diffeotypes $(S^2 \times S^2) \# k\overline{\C\P^2}$, $k=1,2,3$. 
\end{enumerate}
\end{remark}

\subsection*{Structure of the paper:} In Section~\ref{sec:polygon spaces}, we review some facts about polygon spaces and study an $S^1$-action on the pentagon space of interest. In Section~\ref{sec:semitoric}, we review the basic concepts of semitoric systems and study the local theory of the transition point. In Section~\ref{sec:non-deg}, we study some Morse functions on quadrilateral spaces. We apply the main result, Proposition~\ref{prop:rank-1}, to functions induced by $H_t$ on the symplectic reduced spaces $J^{-1}(c)/S^1$ in the proof of Theorem~\ref{thm:main}. In Section~\ref{sec:transition}, we study the transition point and transition times.

\subsection*{Acknowledgements.} We would like to express sincere gratitude to Yohann Le Floch, Joseph Palmer, and Susan Tolman for their patient guidance, insightful discussions, and invaluable feedback during the development of this project. The first author also wants to thank Tara Holm for introducing polygon spaces during his visit to Cornell University, and Allen Knutson for some helpful discussions about pentagon spaces. The first author was partially supported by NSF Grant 2204359.

\section{Polygon spaces}\label{sec:polygon spaces}

In this section, we review some properties of polygon spaces. For a nonempty proper subset $I \subsetneq \{1,\ldots,n\}$, we define $\ell_I: \mathcal{P}(r_1,\ldots,r_n) \to \R$ by $\ell_I([\rho])=|\sum_{i \in I} \rho_i|$. Its Hamiltonian flow is called the \textbf{bending flow} associated to $I$. Whenever $\ell_I$ never vanishes on $\mathcal{P}(r_1,\ldots,r_n)$, the bending flow associated to $I$ is periodic (see~\cite[section 2.1]{Kl94} or ~\cite[Corollary 3.9]{KM96}). It rotates the set of vectors $\{\rho_i : i \in I\}$ around the axis $\sum_{i\in I} \rho_i$. Hence, $\ell_I$ is a moment map of a Hamiltonian $S^1$-action on $\mathcal{P}(r_1,\ldots,r_n)$. We first show that reducing by these $S^1$-actions yields another polygon space. For simplicity, we will pick $I=\{1,2\}$. 

\begin{lemma}\label{lem:reduction}
    Fix positive numbers $r_1> r_2$. Let $M:= S^2_{r_1} \times S^2_{r_2}$ be equipped with the symplectic form $\frac{1}{r_1^2}\omega \oplus \frac{1}{r_2^2} \omega$, where $\omega_u(v,w) = u \cdot (v \times w)$. Consider the Hamiltonian $S^1$-action on $M$ with moment map $f(\rho_1,\rho_2) =|\rho_1+\rho_2|$. For any $c\in [r_1-r_2, r_1+r_2]$, $f^{-1}(c)/S^1$ is a symplectic manifold symplectomorphic to $(S^2_c, \frac{1}{c^2}\omega)$.
\end{lemma}

\begin{proof}
     The Hamiltonian vector field of $f$ is given by $X_f(\rho_1,\rho_2)=\frac{1}{|\rho_1+\rho_2|}(\rho_2 \times \rho_1, \rho_1 \times \rho_2)$. Hence, if $c \in (r_1-r_2, r_1+r_2)$, then $c$ is a regular value of $f$. Moreover, for any $(\rho_1,\rho_2)\in f^{-1}(c)$, $S^1$ acts on it by rotating it around $\rho_1+\rho_2$ (see \cite[Proposition 3.11]{KM96} for example), so $S^1$ acts freely on $f^{-1}(c)$ whenever $c$ is a regular value. It follows that $f^{-1}(c)/S^1$ is a symplectic manifold.

    Consider the map $\Phi:f^{-1}(c) \to S^2_{c}$ that maps $(\rho_1,\rho_2)$ to $\rho_1+\rho_2$. Since $c \in (r_1-r_2,r_1+r_2)$, this map is surjective. Moreover, by the law of cosine, for any $\rho \in S^2_c$, $\Phi^{-1}(\rho)$ is a single $S^1$-orbit.
     We prove that $\Phi^*( \frac{1}{c^2}\omega) = \iota^* (\frac{1}{r_1^2}\omega \oplus \frac{1}{r_2^2} \omega)$, where $\iota: f^{-1}(c) \hookrightarrow M$ is the inclusion.  
     Let $e_1,e_2,e_3$ be the standard basis of $\R^3$ and define $V_i:= (\rho_1 \times e_i,\rho_2 \times e_i)$ for $i=1,2,3$. It is straightforward to check that $V_1,V_2,V_3$ form a basis of $T_{(\rho_1,\rho_2)}f^{-1}(c)$.
    
     \begin{align*}
        (\Phi^*\frac{1}{c^2}\omega)_{(\rho_1,\rho_2)}(V_i,V_j) &=\frac{\rho_1 +\rho_2}{c^2} \cdot [((\rho_1+\rho_2) \times e_i ) \times ((\rho_1+\rho_2) \times e_j)]\\ 
        &= \frac{|\rho_1+\rho_2|^2}{c^2} (\rho_1+\rho_2) \cdot (e_i \times e_j)\\
        & = (\rho_1+\rho_2) \cdot (e_i \times e_j).
     \end{align*}
     On the other hand,
     \begin{align*}
         \iota^*\left(\frac{1}{r_1^2}\omega \oplus \frac{1}{r_2^2} \omega\right)_{(\rho_1,\rho_2)}(V_i,V_j) &=\frac{\rho_1}{r_1^2} \cdot [(\rho_1 \times e_i) \times (\rho_1 \times e_j)]+\frac{\rho_2}{r_2^2} \cdot [(\rho_2 \times e_i) \times (\rho_2 \times e_j)] \\
         &= \frac{|\rho_1|^2}{r_1^2} \rho_1 \cdot(e_i \times e_j) +\frac{|\rho_2|^2}{r_2^2} \rho_2 \cdot(e_i \times e_j) \\
         & = (\rho_1+\rho_2) \cdot (e_i \times e_j).
     \end{align*}
     Therefore, by the symplectic reduction Theorem~\cite{Meyer73,MW74}, when $c$ is a regular value, $f^{-1}(c)/S^1$ is symplectomorphic to $(S^2_c, \frac{1}{c^2}\omega)$.

     Now assume that $c= r_1\pm r_2$. In this case, $X_f(\rho_1,\rho_2) =0$ for any $(\rho_1,\rho_2) \in f^{-1}(c)$, so every point in $f^{-1}(c)$ is a fixed point. It follows that $f^{-1}(c)/S^1 = f^{-1}(c)$. The map $\Psi: S^2_c \to f^{-1}(c)$ sending $\rho$ to $(\frac{r_1}{c}\rho, \pm \frac{r_2}{c}\rho)$ is the inverse of the map $\Phi$ defined above. Moreover, for any $v, w \in T_\rho S^2_c$, we have that 
     
     \begin{align*}
         \Psi^*\left(\frac{1}{r_1^2}\omega \oplus \frac{1}{r_2^2} \omega\right)_{\rho}(v,w) &=\frac{1}{r_1^2} \frac{r_1}{c} \rho \cdot \bigg(\frac{r_1}{c}v \times \frac{r_1}{c}w \bigg) +\frac{1}{r_2^2} \frac{\pm r_2}{c} \rho \cdot \bigg(\frac{\pm r_2}{c}v \times \frac{\pm r_2}{c}w\bigg) \\
         &= \frac{r_1}{c^3} \rho \cdot (v \times w) \pm  \frac{r_2}{c^3} \rho \cdot (v \times w)   \\
         & =  \frac{r_1 \pm r_2}{c^3} \rho \cdot (v \times w) \\
         & = \frac{1}{c^2} \rho \cdot (v \times w) \\
         & =  \frac{1}{c^2} \omega_\rho(v,w).
     \end{align*}
     Hence, $\Psi$ is indeed a symplectomorphism. The claim follows.
\end{proof}

\begin{proposition}\label{prop:reduction}
    Fix positive numbers $r_1,\ldots,r_n$ such that $r_1 \neq r_2$ and $\mathcal{P}(r_1,r_2,\ldots,r_n)$ is a smooth manifold. Fix $|r_1-r_2| \leq c \leq r_1+r_2$ such that the reduced space $\ell_{12}^{-1}(c)/S^1$ is a smooth manifold. Then $\ell_{12}^{-1}(c)/S^1$ is symplectomorphic to $\mathcal{P}(c,r_3,r_4,\ldots,r_n)$
\end{proposition}
\begin{proof}
    Let $\Tilde M: = S^2_{r_1} \times \cdots \times S^2_{r_n}$ be the product of symplectic $2$-spheres of radii $r_1,\ldots,r_n$ and let $\mu: \Tilde M \to \R^3$ be the moment map of the diagonal Hamiltonian $SO(3)$-action on $\Tilde{M}$. Under this setting, $\mathcal{P}(r_1,\ldots,r_n) = \Tilde M\sslash_{0} SO(3) = \mu^{-1}(0)/SO(3)$.
    Notice that there is a Hamiltonian $S^1$-action on $\Tilde{M}$ with moment map $\Tilde{\ell_{12}}(\rho) = |\rho_1+\rho_2|$, and this $S^1$-action commutes with the $SO(3)$-action on $\Tilde{M}$. Hence, we can identify $\ell_{12}^{-1}(c)/S^1$ with $F^{-1}(c,0)/(S^1 \times SO(3))$, where $F = (\Tilde{\ell_{12}}, \mu): \Tilde M \to \R \times \R^3$. It follows that we can do symplectic reductions by stages: first by the $S^1$-action and then by the $SO(3)$-action. In other words, $\ell_{12}^{-1}(c)/S^1 = (\Tilde M \sslash_cS^1)\sslash_0 SO(3)$.

    By Lemma~\ref{lem:reduction}, $\Tilde{\ell_{12}}^{-1}(c)/S^1 \cong S^2_c \times S^2_{r_3} \times \cdots \times S^2_{r_n}$. Hence, by definition of polygon spaces, $\ell_{12}^{-1}(c)/S^1 = (\Tilde M \sslash_cS^1)\sslash_0 SO(3) = (S^2_c \times S^2_{r_3} \times \cdots \times S^2_{r_n})\sslash_0 SO(3) = \mathcal{P}(c,r_3,\ldots,r_n) .$
\end{proof}

Next, we consider the bending flows associated to different index sets. Proposition 2.1.2 in~\cite{Kl94} (or ~\cite[Lemma 2.1]{HT03}) shows that $\{\ell_I^2,\ell_J^2\}=0$ if $I, J$ are disjoint or one is contained in another. Moreover, note that for any $I \subseteq \{1,\ldots,n\}$, $\ell_I = \ell_{I^c}$, where $I^c$ is the complement of $I$ in $\{1,\ldots,n\}$. Together, these imply the following lemma.

\begin{lemma}\label{lem:poisson-commuting}
    Let $I,J$ be nonempty subsets of $\{1,\ldots,n\}$. If one of the following holds: 
    \begin{itemize}
        \item $I, J$ are disjoint, or
        \item $I^c, J^c$ are disjoint, or
        \item one of $I,J$ is contained in another,
    \end{itemize}
    then $\{\ell_I^2,\ell_J^2\} = 0$. Moreover, if $\ell_I,\ell_J$ never vanish on $\mathcal{P}(r_1,\ldots,r_n)$, then $\{\ell_I,\ell_J\} = 0$.
\end{lemma}

Next, we establish a general criterion characterizing the singular points of $H_t$ that will be used in our semitoric families.

\begin{lemma}\label{lem:singular}
    Fix $p \in \mathcal{P}(r_1,\ldots,r_n)$ and choose a representative $\rho \in  \widetilde{\mathcal{P}}(r_1,\ldots,r_n)$ of $p$ with $\rho_1 =(r_1,0,0)$. For any $t \in[0,1]$, $p$ is a singular point of $H_t := t\ell_{n-2,n-1}^2+(1-t)\ell^2_{n-1,n}$ if and only if there exists $a \in \R$ such that 
    \begin{itemize}
        \item $(a,0,0) \times \rho_i =0$ for $i \neq n-2,n-1,n$;
        \item $(a,0,0) \times \rho_{n-2} = t \rho_{n-1} \times \rho_{n-2}$; and
        \item $(a,0,0) \times \rho_n =(1-t)  \rho_{n-1} \times \rho_n$.
    \end{itemize}
\end{lemma}

\begin{proof}
    By~\cite[Lemma 3.5]{KM96}, the Hamiltonian vector field $X_{H_t}$ of $H_t$ on $\widetilde{\mathcal{P}}(r_1,\ldots,r_n)$ is given by 
    \[X_{H_t}(\rho) = (0,\ldots,0, t \rho_{n-1} \times \rho_{n-2}, t \rho_{n-2} \times \rho_{n-1}+ (1-t) \rho_n \times \rho_{n-1},  (1-t) \rho_{n-1} \times \rho_{n} ).\]

    Since the tangent space $T_p\mathcal{P}(r_1,\ldots,r_n)$ is the quotient of $T_\rho\widetilde{\mathcal{P}}(r_1,\ldots,r_n)$ by the vertical space $V_\rho$ generated by the infinitesimal vector fields of the $SO(3)$-action, $p$ is a singular point of $H_t$ if and only if $X_{H_t}(\rho) \in V_\rho$. It is straightforward to compute that 
    \[V_\rho = \{((x,y,z)\times\rho_1, \ldots,(x,y,z) \times \rho_n): x,y,z \in \R\}.\] 
    Hence, $X_{H_t} \in V_\rho$ if and only if there exist $a,b,c \in \R$ such that 
   $(a,b,c) \times \rho_i = 0 \textrm{ for } 1 \leq i \leq n-3$, $(a,b,c) \times \rho_{n-2} = t \rho_{n-1} \times \rho_{n-2}$ and $(a,b,c) \times \rho_n =(1-t)  \rho_{n-1} \times \rho_n$. Since $\rho_1 = (r_1,0,0)$, this implies that $b=c=0$. The result follows immediately.
\end{proof}

\begin{lemma}\label{lem:fixed-components}
Let $r_1,\dots,r_5$ be positive numbers such that the polygon space
$M := \mathcal{P}(r_1,r_2,r_3,r_4,r_5)$ is a smooth manifold, and let $J=\ell_{12}:M\to\mathbb{R}$. Let $P,Q$ be pentagons defined in section~\ref{sec:intro}.
Assume that the side lengths satisfy the inequalities of one of the cases $(P)$, $(Q)$ or $(P+Q)$ in Theorem~\ref{thm:main}.
Then $J$ generates a semi-free Hamiltonian $S^1$-action on $M$, and the fixed-point set of this action always contains exactly two connected fixed surfaces $J^{-1}(|r_1-r_2|)$ and $J^{-1}(r_1+r_2)$.

In addition, depending on which inequalities hold, the action also has isolated fixed points.
\begin{itemize}
    \item In case $(P)$, the set of isolated fixed points is $\{P\}$.
    \item In case $(Q)$, the set of isolated fixed points is $\{Q\}$.
    \item In case $(P+Q)$, the set of isolated fixed points is $\{P,Q\}$.
\end{itemize}
\end{lemma}

\begin{figure}[ht!]
\centering
\begin{subfigure}{0.9\textwidth}
\centering
    \begin{tikzpicture}[scale=0.3]
            % Axes
    \draw[->] (-1,0) -- (9,0); % x-axis
    \draw[->] (-1,0) -- (-1,10); % y-axis

    % Axis labels
    \node[below] at (0,0) {$J_{\min}$};
    \node[below] at (4,0) {$j_P$};
    \node[below] at (8,0) {$J_{\max}$};

    \node[left] at (-1,1) {$r_5-r_4$};
    \node[left] at (-1,4) {$r_3-J_{\min}$};
    \node[left] at (-1,9) {$r_4+r_5$};

    \draw (0,9) -- (8,9) -- (8,1) -- (4,1) -- (0,4) -- (0,9);
    
    % Axes
    \draw[->] (20,0) -- (30,0); % x-axis
    \draw[->] (20,0) -- (20,10); % y-axis

    % Axis labels
    \node[below] at (21,0) {$J_{\min}$};
    \node[below] at (25,0) {$j_P$};
    \node[below] at (29,0) {$J_{\max}$};

    \node[left] at (20,1) {$r_3-r_4$};
    \node[left] at (20,6) {$r_5+J_{\min}$};
    \node[left] at (20,9) {$r_3+r_4$};
    
    % Polygon
    \draw (21,1) -- (29,1) -- (29,9) -- (25,9) -- (21,6) -- (21,1);
\end{tikzpicture}
\caption{Case (P)}
\label{fig:P-image}
\end{subfigure}\\
\begin{subfigure}{0.9\textwidth}
\centering
    \begin{tikzpicture}[scale=0.3]
        % Axes
    \draw[->] (-1,0) -- (9,0); % x-axis
    \draw[->] (-1,0) -- (-1,10); % y-axis

    % Axis labels
    \node[below] at (0,0) {$J_{\min}$};
    \node[below] at (4,0) {$j_Q$};
    \node[below] at (8,0) {$J_{\max}$};

    \node[left] at (-1,1) {$r_5-r_4$};
    \node[left] at (-1,3) {$r_3+J_{\min}$};
    \node[left] at (-1,9) {$r_4+r_5$};

    % Polygon
    \draw (0,1) -- (8,1) -- (8,9) -- (4,9) -- (0,3) -- (0,1);
    
        % Axes
    \draw[->] (20,0) -- (30,0); % x-axis
    \draw[->] (20,0) -- (20,10); % y-axis

    % Axis labels
    \node[below] at (21,0) {$J_{\min}$};
    \node[below] at (25,0) {$j_Q$};
    \node[below] at (29,0) {$J_{\max}$};

    \node[left] at (20,1) {$r_3-r_4$};
    \node[left] at (20,7) {$r_5-J_{\min}$};
    \node[left] at (20,9) {$r_3+r_4$};

    \draw (21,9) -- (29,9) -- (29,1) -- (25,1) -- (21,7) -- (21,9);
\end{tikzpicture}
\caption{Case (Q)}
\end{subfigure}\\
\begin{subfigure}{0.9\textwidth}
\centering
    \begin{tikzpicture}[scale=0.3]
     \draw[->] (-1,0) -- (9,0); % x-axis
    \draw[->] (-1,0) -- (-1,10); % y-axis

    % Axis labels
    \node[below] at (0,0) {$J_{\min}$};
    \node[below] at (2,0) {$j_P$};
    \node[below] at (4,0) {$j_Q$};
    \node[below] at (8,0) {$J_{\max}$};

    \node[left] at (-1,1) {$r_5-r_4$};
    \node[left] at (-1,5) {$r_3-J_{\min}$};
    \node[left] at (-1,7) {$r_3+J_{\min}$};
    \node[left] at (-1,9) {$r_4+r_5$};

    \draw (2,9) -- (8,9) -- (8,1) -- (4,1) -- (0,5) -- (0,7) -- (2,9);
    
    % Axes
    \draw[->] (20,0) -- (30,0); % x-axis
    \draw[->] (20,0) -- (20,10); % y-axis

    % Axis labels
    \node[below] at (21,0) {$J_{\min}$};
    \node[below] at (25,0) {$j_Q$};
    \node[below] at (23,0) {$j_P$};
    \node[below] at (29,0) {$J_{\max}$};

    \node[left] at (20,1) {$r_3-r_4$};
    \node[left] at (20,3) {$r_5-J_{\min}$};
    \node[left] at (20,5) {$r_5+J_{\min}$};
    \node[left] at (20,9) {$r_3+r_4$};

    % Polygon
    \draw (23,1) -- (29,1) -- (29,9) -- (25,9) -- (21,5) -- (21,3) -- (23,1);
\end{tikzpicture}
\caption{Case (P+Q), where $r_3>r_5$}
\end{subfigure}
\caption{Each subfigure contains two moment images of $\mathcal{P}(r_1,\ldots,r_5)$ with moment map $(\ell_{12},\ell_{45})$ on the left and moment map $(\ell_{12},\ell_{34})$ on the right, where $r_1,\ldots,r_5$ satisfy inequalities in the corresponding case. In all figures, $J_{\min} = |r_1-r_2|$, $J_{\max} = r_1+r_2, \  j_P=r_3+r_4-r_5, \ j_Q = -r_3+r_4+r_5$.}
\label{fig:moment-image}
\end{figure}
\begin{proof}
    The proof of all three cases are nearly identical, so we will only prove the first case. Since in all cases, the inequalities imply $r_1 \neq r_2, r_3 \neq r_4, r_4 \neq r_5$, so by Lemma~\ref{lem:poisson-commuting} and \cite[Proposition 3.11]{KM96}, $(\ell_{12},\ell_{45})$ or $(\ell_{12},\ell_{34})$ is the moment map of an effective Hamiltonian $T^2$-action on $M$. It is straightforward to check that the Delzant polygons corresponding to these toric actions are given by Figure~\ref{fig:P-image}. Since $(\ell_{12},\ell_{34})(P) = (j_P,r_3+r_4)$, the claims about fixed components follow directly from Delzant's classification~\cite{De1988} and the local normal form theorem~\cite{GS84}. 
\end{proof}

\section{Semitoric Systems}\label{sec:semitoric}
In this section, we recall the definition of semitoric systems and the classification of non-degenerate singularities. The non-degeneracy of singularities can be defined in any dimension for any integrable system. Since our main interest lies in dimension 4, we will follow Section 2.3 in~\cite{LFP-fam2} to give the equivalent definition using Hessian matrices. For the definition in any dimension using Cartan subalgebras, we refer the readers to~\cite[Definition 1.23]{Bolsinov-Fomenko}.

\subsection{Singularities of $4$-dimensional integrable systems}
Let $(M,\omega)$ be a $4$-dimensional symplectic manifold. Recall that an integrable system on $(M,\omega)$ is a pair of functions $F=(J,H): M \to \R^2$ such that $dJ,dH$ are linearly independent almost everywhere and $\{J,H\}=0$. We say a point $p \in M$ is {\bf singular} if $dJ_m$ and $dH_m$ are linearly dependent; the {\bf rank} of the singular point $p$ is the dimension of the vector space spanned by $dJ_m,dH_m$.  

\subsubsection*{Rank zero points}
Let $p \in M$ be a rank zero singular point and fix a basis of $T_pM$. Let $\Omega_p$ be the matrix of the symplectic form $\omega_p$ and $\Hess(J)_p,\ \Hess(H)_p$ be the Hessian matrices under this basis. For any $\nu,\mu \in \R$, the characteristic polynomial of the matrix $A_{\nu,\mu}:= \Omega_p^{-1}(\nu\Hess(J)_p+\mu\Hess(H)_p)$ is of the form $X \mapsto \chi_{\nu,\mu}(X^2)$ for some quadratic polynomial $\chi_{\nu,\mu}$ (c.f. Proposition 1.2 and the discussion below Theorem 1.3 in~\cite{Bolsinov-Fomenko}). 

\begin{definition}\label{def:rank-0}
    Let $p \in M$ be a rank zero singular point of the integrable system $(J,H)$ on a 4-dimensional symplectic manifold $(M,\omega)$. $p$ is {\bf non-degenerate} if and only if there exist $\nu,\mu \in \R$ such that $\chi_{\nu,\mu}$ has two distinct nonzero roots $\lambda_1,\lambda_2 \in \C$. In this case, the type of $p$ is 
    \begin{itemize}
        \item \textit{elliptic-elliptic} if $\lambda_1,\lambda_2<0$;
        \item \textit{elliptic-hyperbolic} if $\lambda_1<0,\lambda_2>0$; 
        \item \textit{hyperbolic-hyperbolic} if $\lambda_1,\lambda_2>0$; and
        \item \textit{focus-focus} if both $\lambda_1,\lambda_2$ are not purely real numbers, i.e. $\mathcal{I}(\lambda_1) \neq 0$ and $\mathcal{I}(\lambda_2) \neq 0$.
    \end{itemize}
    
\end{definition}

\subsubsection*{Rank one points}
Let $p \in M$ be a rank one singular point. In this case, $dJ_p$ and $dH_p$ are linearly dependent, so the Hamiltonian vector fields $X_J(p),X_H(p)$ are linearly dependent as well because of the non-degeneracy of $\omega$. Let $L$ be the real span of $X_J(p),X_H(p)$ and $L^\omega$ be the symplectic orthogonal of $L$. Then $A_{\nu,\mu}:= \Omega_p^{-1}(\nu\Hess(J)_p+\mu\Hess(H)_p)$ descends to an operator, by slight abuse of notation, still called $A_{\nu,\mu}$ on the 2-dimensional quotient $L^\omega/L$, and the eigenvalues of $A_{\nu,\mu}$ are of the form $\pm \lambda$ for $\lambda \in \C$ (again, see Proposition 1.2 in~\cite{Bolsinov-Fomenko}). 

\begin{definition}
    Let $p \in M$ be a rank one singular point of the integrable system $(J,H)$ on a 4-dimensional symplectic manifold $(M,\omega)$. $p$ is {\bf non-degenerate} if and only if $A_{\nu,\mu}$ is an isomorphism on the quotient space $L^\omega/L$. In this case, the type of $p$ is 
    \begin{itemize}
        \item \textit{elliptic-regular} if $A_{\nu,\mu}$ has eigenvalues $\pm i\alpha$ for some $\alpha \in \R \setminus \{0\}$;
        \item \textit{hyperbolic-regular} if $A_{\nu,\mu}$ has eigenvalues $\pm \alpha$ for some $\alpha \in \R \setminus \{0\}$.
    \end{itemize}
    
\end{definition}

Now, we are ready to define semitoric systems.
\begin{definition}\label{def:semitoric}
    A {\bf semitoric system} is a 4-dimensional integrable system $(M,\omega,F=(J,H))$ such that 
    \begin{enumerate}
        \item $J$ is proper;
        \item $J$ is the moment map for an effective Hamiltonian $S^1$-action; and
        \item all singularities of $F$ are non-degenerate without hyperbolic types.
    \end{enumerate}
\end{definition}
Notice that (3) is equivalent to saying that all rank zero singular points are of either elliptic-elliptic or focus-focus type, and all rank one singular points are of elliptic-regular type.
\subsection{Local normal forms of singularities of a semitoric system}
In this section, we recall the local theory of singularities of a semitoric system. The best-known local theory is proposed by Eliasson~\cite{Eli84}: it is a symplectic Morse lemma for \textit{any} integrable system. However, since we are studying semitoric systems only, the local normal form given by the Hamiltonian $S^1$-action comes in handy, so we will use the following lemma. We refer interested readers to Section 2.2 in~\cite{LFP-fam2} for a discussion of Eliasson's normal form and the references therein.

\subsubsection{Rank zero points} We first discuss the local forms around rank zero points. Notice that these points are all fixed by the $S^1$-action.

\begin{lemma}\label{lem:LFP}\cite[Proposition 7.5]{LFP-fam2}
    Let $J = \epsilon (\frac{1}{2}|z_1|^2 - \frac{1}{2}|z_2|^2)$ with $\epsilon \in \{-1,1\}$. Then $H \in \mathcal{C}^\infty(\C^2,\R)$ is a Hamiltonian such that $\{J, H\}=0$, $H(0) =0$ and $dH(0)=0$ if and only if there exist $\mu_1,\mu_2,\mu_3,\psi \in \R$ and $R_3 \in \mathcal{C}^\infty(\C^2,\R)$ with $R_3(z_1,z_2) = O(3)$ and $\{J,R_3\}=0$ such that 
    \[H(z_1,z_2) = \mu_1 \mathfrak{R}(e^{i\psi}z_1z_2) +\mu_2|z_1|^2+\mu_3|z_2|^2 +R_3(z_1,z_2).\]
    Moreover, if $H$ is of this form and if $(\mu_1,\mu_2+\mu_3) \neq (0,0)$, then $(J,H)$ is integrable. In this case, the singular point $(0,0)$ of $(q_1,H)$ is 
    \begin{itemize}
        \item of focus-focus type if $|\mu_2+\mu_3| < |\mu_1|$;
        \item of elliptic-elliptic type if $|\mu_2+\mu_3| > |\mu_1|$; and
        \item degenerate if $|\mu_2+\mu_3| = |\mu_1|$.
    \end{itemize}
    In this statement $O(3)$ means $O(||z_1,z_2||^3)$; we will use this notation throughout the paper.
\end{lemma}

For a point $p \in M$ fixed by the $S^1$-action, the isotropy weights are the weight vectors of the isotropy representation of $S^1$ on $T_pM$. In dimension $4$, there are two (counted with multiplicity) isotropy weights at each fixed point. The above lemma discusses the case when the isotropy weights are $\pm 1$. In fact, this is the only case in which $p$ can be a focus-focus singularity of the semitoric system. Another case of interest is when $p$ lies on a surface fixed by the $S^1$-action. Since the action is effective, the isotropy weights must be $0,1$ or $0,-1$. We prove the following lemma assuming that $p$ has isotropy weights $0,1$. Notice that we could have used the complex coordinates $z_1,z_2$ and put $J$ as $|z_1|^2$, but this will put $H$ in a complicated form. Since the non-degeneracy and the type of singularity only depend on the second-order terms in the Taylor expansion, we use the real coordinates to simplify the statement and the proof.

\begin{lemma}\label{lem:rank-0-on-fixed-surface}
    Consider the symplectic vector space $(\R^4, \omega = dx_1\wedge dy_1 + dx_2\wedge dy_2)$. Let $J(x_1,y_1,x_2,y_2) = x_1^2+y_1^2$ and $H \in \mathcal{C}^{\infty}(\R^4,\R)$ be a smooth function. Then $\{J,H\} =0$, $H(0) =0$ and $dH(0) =0$ if and only if there exist $a,b,c,d \in \R$, $R_3 \in \mathcal{C}^\infty(\R^4,\R)$ with $R_3 = O(3)$ such that $H =a (x_1^2+y_1^2) + bx_2^2 + cx_2y_2+dy_2^2+R_3(x_1,y_1,x_2,y_2)$ and $\{J,R_3\}=0$. Moreover, in this case, $(0,0)$ is an elliptic-elliptic (elliptic-hyperbolic) point of $(J,H)$ if and only if $(0,0)$ is an elliptic (hyperbolic) singular point of $H|_{J=0}$ in the Morse sense.
\end{lemma}
\begin{proof}
    It is clear that if $H$ is in the given form, $\{J,H\}=0$, $H(0)=0$ and $dH(0)=0$. Conversely, we take the Taylor expansion of $H$. Since $H(0)=0$ and $dH(0)=0$, $H$ has no constant or linear terms. Hence, we can write $H$ as 
    \[H= \sum_{i \leq j} a_{ij}x_ix_j + \sum_{i,j} b_{ij}x_iy_j + \sum_{i \leq j} c_{ij}y_iy_j +R_3(x_1,y_1,x_2,y_2).\]
    It is straightforward to check that
    \[\{J,H\} = 2b_{11}x_1^2+2b_{21}x_1x_2+4(c_{11}-a_{11})x_1y_1+2c_{12}x_1y_2 -2a_{12}x_2y_1-2b_{11}y_{11}^2-2b_{12}y_1y_2 + \{J,R_3\}. \]
    Since $\{J,R_3\}$ contains terms of order at least $3$, $\{J,H\}=0$ implies that 
    \[a_{11}=c_{11},\ a_{12}=b_{11}=b_{12}=b_{21}=c_{12}=0, \textrm{ and } \{J,R_3\}=0.\]    
    Therefore, $H = a_{11}(x_1^2+y_1^2)+a_{22}x_2^2+b_{22}x_2y_2+c_{22}y_2^2+R_3(x_1,y_1,x_2,y_2)$ and $\{J,R_3\}=0$.

    In this case, we have the following matrices:
    \[\Omega = \begin{pmatrix}
        \mathbf{J} & \mathbf{O}\\ \mathbf{O} & \mathbf{J}
    \end{pmatrix}, \
    \Hess(J) = \begin{pmatrix}
        2\mathbf{I} & \mathbf{O}\\ \mathbf{O} & \mathbf{O}
    \end{pmatrix}, \ \text{and} \
    \Hess(H) = \begin{pmatrix}
        2a_{11}\mathbf{I} & \mathbf{O}\\ \mathbf{O} & \mathbf{D}
    \end{pmatrix},  \]
    where $\mathbf{J} = \begin{pmatrix}
        0 & 1 \\ -1 & 0
    \end{pmatrix}$, $\mathbf{D} = \begin{pmatrix}
        2a_{22} & b_{22} \\ b_{22} & 2c_{22}
    \end{pmatrix}$, $\mathbf{O}$ is the zero matrix, $\mathbf{I}$ is the identity matrix, and $\Omega$ is the matrix of the symplectic form in these coordinates.
    Fix a number $\lambda \neq -a_{11}$ and consider $A_{\lambda,1} = \Omega^{-1}(\lambda\Hess(J)+\Hess(H))$. The characteristic polynomial of $A_{\lambda,1}$ is $[x^2+(2\lambda+2a_{11})^2][x^2+4a_{22}c_{22}-{b_{22}}^2]$. Hence, $\chi_{\lambda,1} (X)= (X+(2\lambda+2a_{11})^2)(X+4a_{22}c_{22}-{b_{22}}^2)$. The two roots are two real numbers $\lambda_1 = -(2\lambda+2a_{11})^2$ and $\lambda_2={b_{22}}^2-4a_{22}c_{22}$. Since $\lambda \neq -a_{11}$, we have $\lambda_1<0$. By Definition~\ref{def:rank-0}, $(0,0)$ is elliptic-elliptic if and only if $\lambda_2<0$, and elliptic-hyperbolic if and only if $\lambda_2>0$. Finally, we notice that $\mathbf{D}$ is the Hessian matrix of $H|_{J=0}$, so $(0,0)$ is an elliptic (respectively, hyperbolic) singular point in the Morse sense if and only if $\det(\mathbf{D})=-\lambda_2>0$ (respectively, $\det(\mathbf{D})=-\lambda_2<0$). The claim follows immediately.
\end{proof}

Recall that in a semitoric transition family, the type of singularity of the transition point changes as the parameter varies. The following proposition establishes a criterion for the existence of a transition point, and shows that the transition times are given by the roots of a quadratic polynomial.

\begin{proposition}\label{prop:EE-FF-EE}
    Let $J = \frac{1}{2}|z_1|^2 - \frac{1}{2}|z_2|^2$. Fix $\mu_1,\mu_2,\mu_3,\psi,\nu_2,\nu_3 \in \R$ and $Q_3,R_3 \in \mathcal{C}^\infty(\C^2,\R)$ with $Q_3(z_1,z_2) = O(3), R_3(z_1,z_2) = O(3)$ and $\{J,Q_3\}= \{J,R_3\}=0$. Define $H_0(z_1,z_2) := \nu_2|z_1|^2 + \nu_3|z_2|^2+ Q_3(z_1,z_2)$ and $H_1(z_1,z_2):= \mu_1 \mathfrak{R}(e^{i\psi}z_1z_2) +\mu_2|z_1|^2+\mu_3|z_2|^2 +R_3(z_1,z_2)$. Suppose the following holds:
    \begin{itemize}
        \item $H_0(0,0)$ is a non-degenerate local minimum in the Morse-Bott sense;
        \item $H_1(0,0)$ is a non-degenerate local maximum in the Morse-Bott sense; and
        \item $(\mu_2+\mu_3)^2 > \mu_1^2 > 0$.
    \end{itemize}    
    Then for $t \in [0,1]$, $(J,H_t:=(1-t)H_0+tH_1)$ is integrable, and there exist $0< t^- <t^+ <1$ such that the singular point $(0,0)$ of $(J,H_t)$ is
    \begin{itemize}
        \item of elliptic-elliptic type when $t \in [0,t^-) \cup (t^+,1]$;
        \item of focus-focus type when $t\in(t^-,t^+)$; and
        \item degenerate when $t=t^-,t^+$.
    \end{itemize}
\end{proposition}

\begin{proof}
Since $H_0$ has a non-degenerate local minimum (in the Morse-Bott sense) at the origin, the quadratic form $\nu_2|z_1|^2 + \nu_3|z_2|^2$ must be positive semi-definite and nonzero. This implies that $\nu_2+\nu_3>0$. Moreover, since $H_1$ has a non-degenerate local maximum (in the Morse-Bott sense) at the origin, $\Hess(H_1)$ has non-positive eigenvalues with at least one strictly negative eigenvalue, so the sum of eigenvalues of $\Hess(H_1)$ is strictly negative. Since $\text{det}\big(\lambda I-\text{Hess}(H_1)  \big)=\big(\lambda^2-2(\mu_2+\mu_3)\lambda+4 \mu_2 \mu_3 -\mu_1^2  \big)^2$, it follows that $\mu_2+\mu_3<0$.

Consider the function $H_t=(1-t)H_0+tH_1=t\mu_1\mathfrak{R}(e^{i\psi} z_1 z_2)+ \big(t\mu_2+(1-t)\nu_2 \big)|z_1|^2 + \big(t\mu_3+(1-t)\nu_3 \big)|z_2|^2+O(3)$. When $t \neq 0$, we have $t\mu_1 \neq 0$, so by Lemma~\ref{lem:LFP} $(J,H_t)$ is integrable for $t \in (0,1]$. Since $\nu_2+\nu_3 >0$, $(J,H_0)$ is integrable by the same lemma. Moreover, by Lemma~\ref{lem:LFP} the singularity type of $(J,H_t)$ at $(0,0)$ is entirely determined by the sign of 
\[f(t):=\big(t\mu_2+(1-t)\nu_2 +t\mu_3+(1-t)\nu_3  \big)^2-(t\mu_1)^2.\]
The point is elliptic–elliptic when $f(t)>0$, focus–focus when $f(t)<0$, and degenerate when $f(t)=0$. Thus, the problem is reduced to analyzing the  function  $f(t)$ on the interval $[0,1]$. 

So far, we have the following inequalities:
\begin{enumerate}
    \item $(\mu_2+\mu_3)^2>\mu_1^2$,
    \item $\mu_1^2 > 0$,
    \item $\nu_2+\nu_3>0$, and
    \item $\mu_2+\mu_3<0$.
\end{enumerate}
We can rewrite $f(t)$ as the following quadratic function in $t$:
\[f(t)=\{\big((\mu_2+\mu_3)-(\nu_2+\nu_3) \big)^2-\mu_1^2\}t^2+ 2\big((\mu_2+\mu_3)-(\nu_2+\nu_3)\big)(\nu_2+\nu_3)t +(\nu_2+\nu_3)^2.\]
By (1),(3), and (4), the coefficient of $t^2$ in $f(t)$, which is $\big((\mu_2+\mu_3)-(\nu_2+\nu_3) \big)^2-\mu_1^2=\{(\mu_2+\mu_3)^2-\mu_1^2\}-2(\mu_2+\mu_3)(\nu_2+\nu_3)+(\nu_2+\nu_3)^2$, is strictly positive. By (2) and (3), its discriminant $\De = 4\mu_1^2(\nu_2+\nu_3)^2$ is strictly positive. Hence $f(t)$ is a concave up parabola with two distinct real roots. Finally, we check their position relative to the interval $[0,1]$. At the left endpoint $t=0$, by (3) and (4), we have
\[f(0)=(\nu_2+\nu_3)^2>0 \quad \text{and} \quad f'(0)=2\big((\mu_2+\mu_3)-(\nu_2+\nu_3)\big)(\nu_2+\nu_3)<0,\] 
At the right endpoint $t=1$, by (1), (3), and (4), we have
\[f(1)=(\mu_2+\mu_3)^2-\mu_1^2 > 0 \quad \text{and} \quad f'(1)= 2[\{(\mu_2+\mu_3)^2-\mu_1^2 \}-(\mu_2+\mu_3)(\nu_2+\nu_3)]>0.\]

Hence there exist $t^-, t^+$  such that $0< t^{-} < t^{+} < 1$ with $f(t^{\pm})=0$, and we have
\[
f(t) > 0 \ \text{ for } \ t \in [0,t^{-}) \cup (t^{+},1], \qquad
f(t) < 0 \ \text{ for } \ t \in (t^{-}, t^{+}).
\]
By Lemma~\ref{lem:LFP}, $(0,0)$ is a singular point of elliptic-elliptic type for \(t \in [0,t^{-}) \cup (t^{+},1]\), of focus-focus type for \(t \in (t^{-}, t^{+})\), and degenerate at \(t^{\pm}\). This shows the claimed transition between elliptic-elliptic and focus-focus singularity.
\end{proof}

\subsubsection{Rank one points}
Recall that a point $p\in M$ has rank one if $dJ_p$ and $dH_p$ span a one-dimensional vector space. When $dJ_p \neq 0$, $p$ is not a fixed point of the $S^1$-action. If $S^1$ acts freely on $p$, we can apply symplectic reduction (at least locally) at the level $J(p)$ and study the type of singularity of $[p] \in M^{\textrm{red},J(p)} = J^{-1}(J(p))/S^1$ of the reduced function $H^{\textrm{red},J(p)}$ to which $H$ descends.
\begin{lemma}\label{lem:LFP-rk1}\cite[Lemma 2.6]{LFP24}
    Let $(M,\omega, F=(J,H))$ be a 4-dimensional integrable system, where $J$ is the moment map of an effective Hamiltonian $S^1$-action on $M$. Let $p \in M$ be a point on which $S^1$ acts freely. Let $j= J(p)$. Then $p$ is a rank one singular point of $F$ if and only if $[p]$ is a singular point of $H^{\textrm{red},j}$. Moreover, $p$ is a non-degenerate singular point of $F$ if and only if $[p]$ is non-degenerate for $H^{\textrm{red},j}$ (in the Morse sense) and the type of $p$ is elliptic-regular (respectively hyperbolic-regular) if and only if $[p]$ is an elliptic (respectively hyperbolic) singular point of $H^{\textrm{red},j}$.
\end{lemma}

When $dJ_p=0$ and $dH_p \neq 0$, $p$ lies on a fixed surface. In this case, $p$ is automatically an elliptic-regular singularity.

\begin{lemma}\label{lem:rank-1-on-fixed-surface}
    Consider the symplectic vector space $(\C^2, \omega = \frac{i}{2}(dz_1\wedge d\bar{z_1} + dz_2\wedge d\bar{z_2}))$. Let $J(z_1,z_2) = |z_1|^2$ and $H \in \mathcal{C}^{\infty}(\C^2,\R)$ be a smooth function. Then $\{J,H\} =0$, $H(0) =0$ and $dH(0) \neq 0$ if and only if there exist $a,b \in \R$ not both zero, $R_2 \in \mathcal{C}^\infty(\C^2,\R)$ with $R_2 = O(2)$ such that $H =a \mathcal{R}(z_2) + b\mathcal{I}(z_2) + R_2(z_1,z_2)$ and $\{J,R_2\}=0$. Moreover, in this case, $(0,0)$ is an elliptic-regular point of $(J,H)$.
\end{lemma}
\begin{proof}
    The proof of the first claim is similar to that of Lemma~\ref{lem:rank-0-on-fixed-surface}, so we leave it to the readers to verify. Since $X_J(0)=0$, the real span of $X_J(0)$ and $X_H(0)$ is the same as the real span of $X_H(0)$. Let $L$ denote this real span. Then $L^\omega = \R(\frac{\partial}{\partial z_1}, \frac{\partial}{\partial \overline{z}_1}, X_H(0))$. Hence, we can use the equivalent classes of $\frac{\partial}{\partial z_1}$ and $\frac{\partial}{\partial \overline{z}_1}$ as a basis of $L^{\omega}/L$. Under this basis, we have the following matrix representation:
    \[\Omega = \begin{pmatrix}
        0 & \frac{i}{2}\\ -\frac{i}{2} & 0
    \end{pmatrix} \quad \text{and} \quad \Hess(J) = \begin{pmatrix}
        0 & 1\\ 1 & 0
    \end{pmatrix}. \]
    Hence, $A_{1,0} = \Omega^{-1}\Hess(J)$ has eigenvalues $\pm2i$, so by definition $(0,0)$ is an elliptic-regular point of $(J,H)$.
\end{proof}

Lemma~\ref{lem:rank-0-on-fixed-surface} and Lemma~\ref{lem:rank-1-on-fixed-surface} together imply the following non-degeneracy result for points on a fixed surface.
\begin{proposition}\label{prop:singular-points-on-fixed-surface}
    Let $(M,\omega, (J,H))$ be a 4-dimensional integrable system, where $J$ is the moment map of an effective Hamiltonian $S^1$-action on $M$. Let $\Sigma \subset M^{S^1}$ be a fixed surface. Then any rank one singular point $p \in \Sigma$ is of elliptic-regular type. Moreover, a rank zero singular point $p\in \Sigma$ is  of elliptic-elliptic type (elliptic-hyperbolic type, respectively) if and only if $p$ is an elliptic (hyperbolic resp.) singular point of $H|_\Sigma$ in the Morse sense. 
\end{proposition}

\section{Morse Functions on quadrilateral spaces}\label{sec:non-deg}
Recall that Lemma~\ref{lem:LFP-rk1} implies that the non-degeneracy of rank one singularities can be checked by passing to the symplectic reduced spaces. We implement this idea in this section by studying functions on quadrilateral spaces, which, by Proposition~\ref{prop:reduction}, are symplectomorphic to the reduced spaces of interest. The main result, Proposition~\ref{prop:rank-1}, combined with the aforementioned lemma and proposition, will establish the non-degeneracy of singularities of $F_t$. We start by proving a key lemma.

\begin{lemma}\label{lem:equations}
    Let $c,r_3,r_4,r_5$ be positive numbers such that $\mathcal{P}(c,r_3,r_4,r_5)$ is a smooth manifold.
    For $t \in (0,1)$, define $H_t: \mathcal{P}(c,r_3,r_4,r_5) \to \R$ by $H_t([(\rho_0,\rho_3,\rho_4,\rho_5)]) = (1-t)|\rho_4+\rho_5|^2+t|\rho_3+\rho_4|^2$. Then $m \in \mathcal{P}(c,r_3,r_4,r_5)$ is a singular point of $H_t$ if and only if there exists a representative $\rho$ with $\rho_0=(c,0,0), \rho_i=(a_i,b_i,0)$ for $i=3,4,5$ and a nonzero number $a \in \R$ such that 
    \begin{equation}\label{eqn:singular-condition}
    ab_3 =t(a_4b_3-a_3b_4) \quad \textrm{ and } \quad ab_5 = (1-t)(a_4b_5-a_5b_4). \tag{$*$}
\end{equation}
    In this case, the following holds:
    \begin{enumerate}
        \item $b_3,b_5$ are nonzero. When $t \neq \dfrac{1}{2}$, $b_4$ is also nonzero.
        \item  When $t \neq \dfrac{1}{2}$ or $b_4 \neq 0$, we have $K:= \dfrac{tb_5}{(1-t)b_3} \neq -1$. When $t = \dfrac{1}{2}$ and $b_4 =0$, we have $a_3b_5-a_5b_3\neq0$.
    \end{enumerate}
\end{lemma}
\begin{proof}
    Consider a representative $\rho=(\rho_0,\rho_3,\rho_4,\rho_5) \in \widetilde{\mathcal{P}}(c,r_3,r_4,r_5)$ of $m$ such that $\rho_0 =(c,0,0)$. By Lemma~\ref{lem:singular}, $m$ is a singular point of $H_t$ if and only if there exists  $a \in \R$ such that $(a,0,0) \times \rho_3 = t\rho_4\times \rho_3$ and $(a,0,0) \times \rho_5 = (1-t)\rho_4\times \rho_5$.
    This implies that $\rho_0,\rho_3,\rho_4,\rho_5$ are coplanar, so we can further assume that $\rho_i = (a_i,b_i,0)$ for $i=3,4,5$. If $a=0$, then $\rho_3,\rho_4,\rho_5$ are colinear. Since $\rho_0+\rho_3+\rho_4+\rho_5=0$, this implies that $\rho_0,\rho_3,\rho_4,\rho_5$ are all colinear, which is impossible since $\mathcal{P}(c,r_3,r_4,r_5)$ is a smooth manifold. Hence, we conclude that $m$ is a singular point of $H_t$ if and only if there exists a representative $\rho$ of the form $\rho_0=(c,0,0),\rho_i=(a_i,b_i,0)$ for $i=3,4,5$ such that there exists  $a \in \R \setminus \{0\}$ with 
\begin{equation*}
    ab_3 =t(a_4b_3-a_3b_4) \quad \textrm{ and }  \quad ab_5 = (1-t)(a_4b_5-a_5b_4).
\end{equation*}

    (1) If $b_3 =0$, then the equation $ab_3=t(a_4b_3-a_3b_4)$ implies that $b_4 =0$. However, since $\rho_0+\rho_3+\rho_4+\rho_5=0$, this further implies that $b_5=0$ and thus $\rho_0,\rho_3,\rho_4,\rho_5$ are all colinear, which is impossible, so $b_3 \neq 0$. Similarly, $b_5 \neq 0$.  If $b_4 =0$, then Equation~\eqref{eqn:singular-condition} implies that $a = ta_4 $ and $ a=(1-t) a_4$. Since both $a$ and $a_4$ are nonzero, $t= \frac{1}{2}$. Hence, we conclude that $b_4$ is nonzero whenever $t \neq \dfrac{1}{2}$.
    
    (2) When $t\neq \dfrac{1}{2}$, we prove $K\neq -1$ by contradiction. If $K = -1$, then $b_5=-\dfrac{1-t}{t}b_3$, so $b_4 = -b_3-b_5 = \dfrac{1-2t}{t}b_3 =-\dfrac{1-2t}{1-t}b_5$. 
    Since $ab_3 =t(a_4b_3-a_3b_4)$ and $b_3 \neq 0$, we can substitute $b_4$ by $\dfrac{1-2t}{t}b_3$ and divide by $b_3$ on both sides to get
\[a = t\bigg(a_4 - a_3 \frac{1-2t}{t}\bigg) = ta_4 - (1-2t)a_3.\]
Similarly, using $ab_5 =(1-t)(a_4b_5-a_5b_4)$ and $b_4 = -\dfrac{1-2t}{1-t}b_5$, we obtain
\[a = (1-t)\bigg(a_4 + a_5 \frac{1-2t}{1-t}\bigg) = (1-t)a_4 + (1-2t)a_5.\]
These two equations imply that $(1-2t)(a_3+a_4+a_5) = 0$. If $t\neq \dfrac{1}{2}$, $a_3+a_4+a_5=0$. This is impossible as $c+a_3+a_4+a_5=0$ and $c \neq 0$. Therefore, $K \neq -1$ whenever $t \neq \dfrac{1}{2}$.

    When $t =\dfrac{1}{2}$, $K = \dfrac{b_5}{b_3}$. Since $b_3+b_4+b_5=0$, we have $1+\dfrac{b_4}{b_3}+K=0$. Hence, if $b_4 \neq 0$, $K \neq -1$. If $b_4 =0$, then $b_3 = -b_5$, so $a_3b_5-a_5b_3 = b_5(a_3+a_5)$. Assume for the sake of contradiction that $b_5(a_3+a_5)=0$, then $a_3= - a_5$, because $b_5$ is nonzero. This implies $r_3= r_5$. Furthermore, since $c+a_3+a_4+a_5 =0$, we have $c+a_4=0$. Since $b_4=0$, $|a_4|= r_4$, so $c= r_4$. However, this means that $c+r_3-r_4-r_5=0$, contradicting the assumption that $\mathcal{P}(c,r_3,r_4,r_5)$ is a smooth manifold. Hence, we conclude that $a_3b_5-a_5b_3 \neq 0$.
\end{proof}

\begin{proposition}\label{prop:rank-1}
    Let $c,r_3,r_4,r_5$ be positive numbers such that $M:= \mathcal{P}(c,r_3,r_4,r_5)$ is a smooth manifold. For $t \in (0,1)$, define $H_t: \mathcal{P}(c,r_3,r_4,r_5) \to \R$ by $H_t([(\rho_0,\rho_3,\rho_4,\rho_5)]) = (1-t)|\rho_4+\rho_5|^2+t|\rho_3+\rho_4|^2$. Then $H_t$ is a Morse function with only elliptic singularities.
\end{proposition}

\begin{proof}
    Let $m \in M$ be a singular point of $H_t$. By Lemma~\ref{lem:equations}, there exists a representative $\rho$ of $m$ with $\rho_0=(c,0,0), \rho_i=(a_i,b_i,0)$ for $i=3,4,5$ and a nonzero number $a \in \R$ such that 
    \begin{equation}
    ab_3 =t(a_4b_3-a_3b_4) \quad \textrm{ and }  \quad ab_5 = (1-t)(a_4b_5-a_5b_4).\tag{$*$}
\end{equation}
    To compute the Hessian at $m$, we find $W_1,W_2 \in T_\rho\widetilde{\mathcal{P}}(c,r_3,r_4,r_5)$ such that $[W_1],[W_2]$ form a basis of $T_mM$ and then find the matrix representation of the Hessian of $H_t$ under this basis.
Since $\widetilde{\mathcal{P}}(c,r_3,r_4,r_5)$ is a submanifold of $S^2_c \times S^2_{r_3} \times S^2_{r_4} \times S^2_{r_5}$, it suffices to find a coordinate chart of $S^2_c \times S^2_{r_3} \times S^2_{r_4} \times S^2_{r_5}$ around $\rho$. By item (1) of Lemma~\ref{lem:equations}, we have $b_3,b_5 \neq 0$. Depending on whether $b_4$ is zero, we choose different charts.

{\bf Case 1: when $b_4 \neq 0$, we choose the charts with coordinates $(y_0,z_0,x_3,z_3,x_4,z_4,x_5,z_5)$.}
In this chart, the infinitesimal vector field at $\rho$ generated by the $SO(3)$-action is 
\[V_1=b_3\frac{\partial}{\partial z_3}+b_4\frac{\partial}{\partial z_4}+b_5\frac{\partial}{\partial z_5}  , \ V_2=-c\frac{\partial}{\partial z_0}-a_3\frac{\partial}{\partial z_3}-a_4\frac{\partial}{\partial z_4}-a_5\frac{\partial}{\partial z_5}  , \ V_3=c\frac{\partial}{\partial y_0}-b_3\frac{\partial}{\partial x_3}-b_4\frac{\partial}{\partial x_4}-b_5\frac{\partial}{\partial x_5}.\]

Let $V_\rho$ be the real span of $V_1,V_2,V_3$. Define $W_1 =\frac{\partial}{\partial z_3}-\frac{\partial}{\partial z_5}$ and $W_2 =(a_5b_3-a_3b_5)\frac{\partial}{\partial y_0}-b_3b_5\frac{\partial}{\partial x_3}+b_3b_5\frac{\partial}{\partial x_5}$. 
It is straightforward to verify that $W_1,W_2,V_1,V_2,V_3$ form a basis of $T_\rho\widetilde{\mathcal{P}}(c,r_3,r_4,r_5)$ which is the kernel of $d\mu_\rho$, where $\mu: S^2_c \times S^2_{r_3} \times S^2_{r_4} \times S^2_{r_5} \to \R^3$ is the moment map of the $SO(3)$-action. Since $T_mM \cong T_\rho\widetilde{\mathcal{P}}(c,r_3,r_4,r_5)/V_\rho$, $[W_1],[W_2]$ form a basis of $T_mM$. 

Recall that the Hessian of a function $f$ on a manifold $N$ at a singular point $n \in N$ can be computed in the following way:
if $X$ and $Y$ are tangent vectors in $T_nN$, then $\Hess(f)_n(X,Y): = X \cdot (\widetilde{Y} \cdot f)(n)$, where $\widetilde{Y}$ denotes a vector field extending $Y$ locally.

Since $\widetilde{\mathcal{P}}(c,r_3,r_4,r_5)$ is a principal $SO(3)$-bundle over $\mathcal{P}(c,r_3,r_4,r_5)$, 
$\Hess(H_t)_{m}([W_i],[W_j]) =\Hess(H_t)_{\rho}(W_i,W_j)$. We will compute the latter by extending $W_1,W_2$ locally on $\widetilde{\mathcal{P}}(c,r_3,r_4,r_5)$. Let $\widetilde{W}_i$ be a local extension of $W_i$ for $i=1,2$. We can write it as 
\[\widetilde{W}_i = f^i \frac{\partial}{\partial y_0} + g^i \frac{\partial}{\partial z_0} + \sum_{j=3}^5 f_j^i \frac{\partial}{\partial x_j} + \sum_{j=3}^5 g_j^i \frac{\partial}{\partial z_j}.\]
Since $\widetilde{W}_i \in T_\rho\widetilde{\mathcal{P}}(c,r_3,r_4,r_5)=\textrm{ker}(d\mu_\rho)$, the coefficient functions satisfy the following equation:
\begin{equation}\label{eqn:useful}
    -f^i \frac{y_0}{x_0} - g^i \frac{z_0}{x_0} + f_3^i+f_4^i+f_5^i =0.
\end{equation}

It is straightforward to compute that 
\[W_1(\widetilde{W}_iH_1)(\rho) = \frac{a_4b_3-a_3b_4}{b_3}\left( \frac{\partial f_3^i}{\partial z_3}(\rho)-\frac{b_3}{b_4}\frac{\partial f_4^i}{\partial z_3}(\rho)-\frac{\partial f_3^i}{\partial z_5}(\rho)+\frac{b_3}{b_4}\frac{\partial f_4^i}{\partial z_5}(\rho) \right) -g_3^i(\rho)\frac{b_4}{b_3},\]
\[W_1(\widetilde{W}_iH_0)(\rho) = \frac{a_4b_5-a_5b_4}{b_5}\left( \frac{\partial f_5^i}{\partial z_3}(\rho)-\frac{b_5}{b_4}\frac{\partial f_4^i}{\partial z_3}(\rho)-\frac{\partial f_5^i}{\partial z_5}(\rho)+\frac{b_5}{b_4}\frac{\partial f_4^i}{\partial z_5}(\rho) \right) +g_5^i(\rho)\frac{b_4}{b_5}.\]
Since $ab_3 =t(a_4b_3-a_3b_4)$, $ab_5 = (1-t)(a_4b_5-a_5b_4)$, and $b_3+b_4+b_5=0$, we have 
\begin{align*}
    W_1(\widetilde{W}_iH_t)(\rho) & = t W_1(\widetilde{W}_iH_1)(\rho) + (1-t) W_1(\widetilde{W}_iH_0)(\rho) \\
    & =a \left( \frac{\partial (f_3^i+f_4^i+f_5^i)}{\partial z_3}(\rho)-\frac{\partial (f_3^i+f_4^i+f_5^i)}{\partial z_5}(\rho) \right) -tg_3^i(\rho)\frac{b_4}{b_3}+(1-t) g_5^i(\rho)\frac{b_4}{b_5}. 
\end{align*}
By differentiating Equation~\eqref{eqn:useful} with respect to $z_j$ for $j=3,4,5$ and evaluating it at $\rho$, we have $\frac{\partial (f_3^i+f_4^i+f_5^i)}{\partial z_j}(\rho)=0$. Hence, we have
\[ W_1(\widetilde{W}_iH_t)(\rho)  = -tg_3^i(\rho)\frac{b_4}{b_3}+(1-t) g_5^i(\rho)\frac{b_4}{b_5},\]
which evaluates to $-\frac{b_4}{b_3}t-\frac{b_4}{b_5}(1-t)$ when $i=1$ and evaluates to $0$ when $i=2$.
Similarly, one can compute
\[ W_2(\widetilde{W}_2H_t)(\rho)  = \frac{a(a_3b_5-a_5b_3)^2}{c}-t\frac{a_3^2+b_3^2}{b_3}b_4b_5^2 - (1-t) \frac{a_5^2+b_5^2}{b_5}b_4b_3^2.\]
Thus, the matrix representation of the Hessian of $H_t$ at $m$ under the basis $[W_1],[W_2]$ is

\[\begin{pmatrix}
    -t\frac{b_4}{b_3}-(1-t)\frac{b_4}{b_5} & 0 \\
    0 & \frac{a}{c}(a_3b_5-a_5b_3)^2-\frac{t}{b_3}(a_3^2+b_3^2)b_4b_5^2-\frac{1-t}{b_5}(a_5^2+b_5^2)b_3^2b_4
\end{pmatrix}.\]

We show that this matrix has a positive determinant. Set $X=a_4b_3-a_3b_4,\ Y=a_4b_5-a_5b_4$, then $X+Y = a_4(b_3+b_5) -(a_3+a_5)b_4 = -a_4b_4+ (c+a_4)b_4 =cb_4$. Since $c \neq 0$, we can write $b_4 = \frac{X+Y}{c}$. By Equation~\eqref{eqn:singular-condition}, $ab_3 =tX, ab_5 =(1-t)Y$, so we have $X = \frac{ab_3}{t}, Y= \frac{ab_5}{1-t}$. Hence,
\[A:= -t\frac{b_4}{b_3}-(1-t)\frac{b_4}{b_5}  = -b_4 \left(\frac{t}{b_3} + \frac{1-t}{b_5} \right) = -\frac{X+Y}{c} \left( \frac{a}{X} + \frac{a}{Y}\right) = -\frac{a}{c}\left (2+\frac{X}{Y}+\frac{Y}{X} \right ).\]
Let $K := \frac{Y}{X} = \frac{tb_5}{(1-t)b_3}$, then $A= -\frac{a}{cK}(K+1)^2$. Similarly, one can compute
\[B:=\frac{a}{c}(a_3b_5-a_5b_3)^2-\frac{t}{b_3}(a_3^2+b_3^2)b_4b_5^2-\frac{1-t}{b_5}(a_5^2+b_5^2)b_3^2b_4 = -\frac{a}{cK}\left[ (Ka_3b_5+a_5b_3)^2 + (K+1)^2b_3^2b_5^2 \right]. \]

By item (2) of Lemma~\ref{lem:equations}, $K \neq -1$, so both $A$ and $B$ are nonzero. Hence, $\det(\Hess(H_t)_m)= AB =\dfrac{a^2}{c^2K^2}(K+1)^2\left[ (Ka_3b_5+a_5b_3)^2 + (K+1)^2b_3^2b_5^2 \right]>0$. We conclude that $m$ is an elliptic singularity (in the Morse sense).

{\bf Case 2: when $b_4 = 0$, then $a_4 \neq 0$, so we choose the chart with coordinates $(y_0,z_0,x_3,z_3,y_4,z_4,x_5,z_5)$.} Notice that by (1) of Lemma~\ref{lem:equations}, this could only happen when $t = \dfrac{1}{2}$. A similar computation shows that under the basis $W_1 =\frac{\partial}{\partial z_0}-\frac{\partial}{\partial  z_3 }$ and $W_2 = (a_5b_3-a_3b_5)\frac{\partial}{\partial y_0}-b_3b_5\frac{\partial}{\partial x_3}+b_3b_5\frac{\partial}{\partial x_5}$, the matrix representation of the Hessian of $H_t$ is 

\[\begin{pmatrix}
    \dfrac{a}{c}& 0 \\
    0 & \dfrac{a}{c}(a_3b_5-a_5b_3)^2
\end{pmatrix}.\]
By (2) of Lemma~\ref{lem:equations}, we know that $a_3b_5-a_5b_3 \neq 0$, so the determinant of the matrix has the same sign as $\dfrac{a^2}{c^2}>0$. We conclude that $m$ is an elliptic singularity (in the Morse sense).
\end{proof}
\section{The transition points}\label{sec:transition}
In this section, we study the transition points and the transition times. 

\begin{proposition}\label{prop:transition-point}
    Let $r_1,\ldots,r_5$ be positive numbers such that $M:=\mathcal{P}(r_1,r_2,r_3,r_4,r_5)$ is a smooth manifold. Let $P,Q,j_P,j_Q,t_P^\pm,t_Q^\pm$ be as in Section~\ref{sec:intro}. Consider the family $F_t=(J,H_t): = (\ell_{12},t\ell_{34}^2 +(1-t)\ell_{45}^2)$ on $M$, where $t \in [0,1]$. 
    
    If the inequalities in case (P) in Theorem~\ref{thm:main} are satisfied, then
    \begin{itemize}
        \item[(1)] $P$ is a rank zero singular point of $F_t$ which is of elliptic-elliptic type for $t \in [0,t_P^-)\cup(t_P^+,1]$ and of focus-focus type for $t \in (t_P^-,t_P^+)$.
    \end{itemize}
    
    If the inequalities in case (Q) in Theorem~\ref{thm:main} are satisfied, then
    \begin{itemize}
        \item[(2)] $Q$ is a rank zero singular point of $F_t$ which is of elliptic-elliptic type for $t \in [0,t_Q^-)\cup(t_Q^+,1]$ and of focus-focus type for $t \in (t_Q^-,t_Q^+)$.
    \end{itemize}
    
    If the inequalities in case (P+Q) in Theorem~\ref{thm:main} are satisfied, then both (1) and (2) hold.
\end{proposition}

\begin{proof}
    The proofs of all three cases are nearly identical, so we will only prove the result in (1) and discuss how to derive the other cases in the last paragraph of this proof. Assume that the inequalities in case (P) in Theorem~\ref{thm:main} are satisfied. Then $r_5>r_4$ and $j_P>0$.

    By Lemma~\ref{lem:fixed-components}, $P$ is a singular point of $J$. By Lemma~\ref{lem:singular}, $P$ is also a singular point of  $H_t$ for all $t \in [0,1]$, so it is a rank zero singularity of $F_t$. To check the type of singularity, we need to compute the Hessian matrix of  $J$ and $H_t$, together with the matrix of the symplectic form with respect to a chosen basis. Since $\mathcal{P}(r_1,r_2,r_3,r_4,r_5)$ is a sub-quotient of $S^2_{r_1} \times S^2_{r_2} \times S^2_{r_3} \times S^2_{r_4} \times S^2_{r_5}$, we can use the standard coordinates on $(\R^3)^5$ to induce the coordinates on $\mathcal{P}(r_1,r_2,r_3,r_4,r_5)$.
    Since $\rho$ is a representative of $P$ with nonzero $x$-coordinates, $(y_1,z_1,y_2,z_2,y_3,z_3,y_4,z_4,y_5,z_5)$ gives a coordinate chart near $\rho$ on the product of spheres. Notice that $T_P\mathcal{P}(r_1,r_2,r_3,r_4,r_5) \cong T_\rho\Tilde{\mathcal{P}}(r_1,r_2,r_3,r_4,r_5)/V_\rho$, where $V_\rho$ is the space of infinitesimal vector fields generated by the $SO(3)$-action, so it suffices to find vectors in $T_\rho\Tilde{\mathcal{P}}(r_1,r_2,r_3,r_4,r_5)$ that descend to a basis of $T_P\mathcal{P}(r_1,r_2,r_3,r_4,r_5)$.

Define $W_1=\frac{\partial}{\partial z_1}-\frac{\partial}{\partial z_3}$, $W_2=\frac{\partial}{\partial z_1}-\frac{\partial}{\partial z_4}$, $W_3=\frac{\partial}{\partial y_3}-\frac{\partial}{\partial y_4}$, $W_4=\frac{\partial}{\partial y_4}-\frac{\partial}{\partial y_5}$. Notice that $T_\rho\Tilde{\mathcal{P}}(r_1,r_2,r_3,r_4,r_5)$ is the kernel of $d\mu_\rho$, where $\mu:  S^2_{r_1}\times S^2_{r_2} \times S^2_{r_3} \times S^2_{r_4} \times S^2_{r_5} \to \R^3$ is the moment map of the $SO(3)$-action, so $W_i \in T_\rho\Tilde{\mathcal{P}}(r_1,r_2,r_3,r_4,r_5)$ for $i=1,2,3,4$.
It is straightforward to verify that $W_1,W_2,W_3,W_4$ span a 4-dimensional vector space that intersects $V_\rho$ trivially, so $[W_1], [W_2],[W_3],[W_4]$ form a basis of $T_P\mathcal{P}(r_1,r_2,r_3,r_4,r_5)$. Using a similar calculation as in the proof of Proposition~\ref{prop:rank-1}, the Hessian of $H_0$, $H_1$, \text{and} $J$, together with the matrix of the symplectic form $\Omega_P$, under this basis  
are given by
\[ \Hess(H_0)_P=
\begin{pmatrix}
    0 & 0 & 0 & 0\\
    0 & \frac{r_5}{r_4} & 0 & 0\\
    0 & 0 &  \frac{r_5}{r_4}  & 1- \frac{r_5}{r_4} \\
    0 & 0 & 1- \frac{r_5}{r_4} &  \frac{r_5}{r_4} + \frac{r_4}{r_5} -2
\end{pmatrix}, \
\text{Hess}(H_1)_P= \begin{pmatrix}
     -\frac{r_4}{r_3}  & 1 & 0 & 0\\
    1 & -\frac{r_3}{r_4} & 0 & 0\\
    0 & 0 & -2-\frac{r_4}{r_3}-\frac{r_3}{r_4}& 1+ \frac{r_3}{r_4}\\
    0 & 0 & 1+ \frac{r_3}{r_4}& - \frac{r_3}{r_4}\\
\end{pmatrix},
\]
\[\text{Hess}(J)_P=\begin{pmatrix}
    1-\frac{j_P}{r_3} & 1 & 0 & 0\\
    1 & 1-\frac{j_P}{r_4} & 0 & 0\\
    0 & 0 & -j_P(\frac{1}{r_3}+\frac{1}{r_4})& \frac{j_P}{r_4}\\
    0 & 0 & \frac{j_P}{r_4}& j_P(\frac{1}{r_5}-\frac{1}{r_4})\\
\end{pmatrix}, \
\Omega^{-1}_P= \begin{pmatrix}
        0 & 0 & r_3 & r_3 \\
        0 & 0 & 0 & r_4\\
        -r_3 & 0 & 0 & 0 \\
        -r_3 & -r_4 & 0 & 0\\
    \end{pmatrix}.\]

Fix $\lambda \in \R$ such that $\lambda > \max(\frac{r_5-r_4}{j_P}, \frac{r_3r_5}{j_P(j_P+r_4)})$ and consider \begin{align*}
    &A_{\lambda,1}:=\Omega_P^{-1}\big(\lambda\text{Hess}(J)_P+\text{Hess}(H_t)_P \big) \\&=
\begin{pmatrix}
        0 & 0 &  -(2r_3+r_4)t + r_3 - j_P\lambda & \frac{r_3}{r_5}((2r_5-r_4)t+r_4 - r_5  + j_P\lambda  )\\
    0 & 0 & (r_3+r_5)t+ r_4 - r_5  + j_P\lambda   & \frac{-(r_3r_5 + (r_4 - r_5)^2)t + (r_4 - r_5)^2+ \lambda(r_4 - r_5)j_P}{r_5}\\
   r_4t+(r_4-r_5)\lambda & -r_3(t+\lambda) & 0 & 0\\
   -r_5\lambda & r_5(t-\lambda-1)& 0 & 0
    \end{pmatrix}.
\end{align*}

The quadratic polynomial corresponding to its characteristic polynomial is given by $\chi_{\lambda,1}(X)=X^2+A(t)\,X+B(t)$,
where
\[
\begin{aligned}
A(t)&=
[(r_3+r_5)^{2}+2r_4j_P ]t^2
+2\!\left[\lambda (r_3+r_5)j_P-r_3r_5-(r_4-r_5)^2\right]t+\lambda^2j_P^2+(\lambda j_P + r_4-r_5)^2,\\
B(t)&=j_P^{2}\,\bigl(-r_{4}t^{2}+r_{3}\lambda t+r_{4}t+r_{5}\lambda t +j_P\lambda^2+ r_{4}\lambda-r_{5}\lambda\bigr)^{2}.
\end{aligned}
\]
By Definition~\ref{def:rank-0}, $P$ is of elliptic-elliptic type if $\chi_{\lambda,1}(X)$ has two distinct negative roots, i.e. $A(t)>0, B(t) > 0,\; A(t)^2 > 4B(t)$; and $P$ is of focus-focus type if $\chi_{\lambda,1}(X)$ has no purely real roots, i.e. $A(t)^2 < 4B(t)$.

We first compare $A(t)^2$ with $4B(t)$. $A(t)^2-4B(t)=\bigl((r_3+r_5)t+ r_4-r_5+2j_P\lambda\bigr)^{2}\,\bigl(a t^{2}+b t+c\bigr),$
%\[A(t)^2-4B(t)=\bigl((r_3+r_5)t+ r_4-r_5+2j_P\lambda\bigr)^{2}\,\bigl(a t^{2}+b t+c\bigr),\]
where $a= (r_{3}+r_{5})^{2}+4r_{4}j_P, b =-2\bigl(r_4j_P+r_3r_5+(r_4-r_5)^2\bigr), c= (r_{4}-r_{5})^{2}$.
Since $\lambda > \frac{r_5-r_4}{j_P}$, we have $(r_3+r_5)t +r_4-r_5+2j_P\lambda > r_5-r_4 >0$ for all $t \geq 0$. 
Hence, $A(t)^2-4B(t)$ has the same sign as $f(t):=at^2+bt+c$ when $t\in[0,1]$. Since the discriminant of $f(t)$ is $b^{2}-4ac
=16\,r_{3}r_{4}r_{5}j_P>0$ and $a>0$, $f(t)$ has two distinct roots; moreover, by the root formula, the two roots are $t_P^\pm$.
We conclude that 
\begin{itemize}
    \item when $t \in (-\infty,t_P^-)\cup(t_P^+,\infty)$, we have $A(t)^2>4B(t)$, so $\chi_{\lambda,1}(X)$ has two distinct real roots.
    \item when $t \in(t_P^-,t_P^+)$, we have $A(t)^2<4 B(t)$, so $\chi_{\lambda,1}(X)$ has no real roots, and it follows that $P$ is of focus-focus type.
\end{itemize}
Moreover, since $f(0) =c>0 ,f'(0)=b<0, f(1) =(r_3+r_4)^2>0, f'(1) =2(r_3+r_4)^2+2r_4j_P+2r_3r_5 >0$, the two roots $t_P^{\pm}$ of $f(t)$ must both be in $(0,1)$.

We now show that when $t\in[0,t_P^-)\cup(t_P^+,1]$, both roots of $\chi_{\lambda,1}(X)$  are negative by showing that $A(t)>0,B(t)>0$. To check that $B(t)>0$, it suffices to show that $g(t):=-r_{4}t^{2}+r_{3}\lambda t+r_{4}t+r_{5}\lambda t +j_P\lambda^2+ r_{4}\lambda-r_{5}\lambda$ has no roots in $[0,1]$. Since $\lambda >\frac{r_5-r_4}{j_P}>0$, $g(0) = \lambda(j_P\lambda +r_4-r_5)$ and $g(1) = \lambda(j_P\lambda + r_3+r_4)$ are both positive. Since $g$ describes a parabola opening downward, $g(t)$ has no roots in $[0,1]$, i.e., $B(t)= j_P^2g(t)^2 >0$ for all $t\in[0,1]$. As for $A(t)$, we notice that the coefficient of $t^2$ is positive and the discriminant $\Delta$ of $A(t)$ is 
\[\Delta =-\,4j_P\Bigl(
[j_P^3+2r_4j_P^2+(r_4^2+4r_3r_5)j_P]\lambda^2 +2r_4j_P(j_P+r_4)\lambda - 2r_3r_4r_5+r_4^2j_P\Bigr).\]
Since $\lambda > \frac{r_3r_5}{j_P(j_P+r_4)}$, we have $2r_4j_P(j_P+r_4)\lambda  - 2r_3r_4r_5 > 0$, so $\Delta$ is negative and $A(t)$ is always positive. Consequently, $P$ is of elliptic–elliptic type for $t\in[0,t_P^-)\cup(t_P^+,1]$, and of focus–focus type for $t\in(t_P^-,t_P^+)$. This concludes the proof of case (1).

To derive the proof of case (2) from case (1), we first notice that the inequalities in case (Q) imply that $r_3>r_4$ and $j_Q>0$.  Then one replaces $P$ with $Q$, $r_3$ with $-r_3$, and $r_5$ with $-r_5$ in the preceding Hessian computations and sign analysis, and argues in the same manner that the relevant quantities remain positive or negative under the choice  $\lambda > \max(\frac{r_3-r_4}{j_Q},\frac{r_3r_5}{j_Q(r_4+j_Q)})$. 
Finally, the last case can be proved by noticing that the inequalities in case (P+Q) imply that $r_3,r_5>r_4$, $j_P,j_Q>0$ and joining the proof of case (1) and case (2).
\end{proof}

\bibliographystyle{alpha}
\bibliography{ref}

\end{document}